\documentclass [a4paper, 12pt]{amsart}

\usepackage[margin=2.5cm]{geometry} 
\usepackage[alphabetic]{amsrefs}
\usepackage{wasysym,stmaryrd,enumerate}

\usepackage{amsthm}
\usepackage{latexsym}
\usepackage{times}
\usepackage{graphicx}
\usepackage{yhmath}

\usepackage{amsmath}
\usepackage{amsfonts}
\usepackage{amssymb}
\usepackage{amsthm}
\usepackage{graphicx}
\usepackage{pb-diagram}
\usepackage{epstopdf}
\usepackage{fancyhdr}
\usepackage{mathrsfs}
\usepackage[all]{xy}

\newtheorem{theorem}{Theorem}[section] 
\newtheorem{lemma}[theorem]{Lemma}     
\newtheorem{corollary}[theorem]{Corollary}
\newtheorem{proposition}[theorem]{Proposition}
\newtheorem{remark}[theorem]{Remark}
\newtheorem{definition}[theorem]{Definition}
\newtheorem*{theorem*}{Theorem}
\newtheorem*{definition*}{Definition}
\newtheorem{example}[theorem]{Example}

\DefineSimpleKey{bib}{a}
\BibSpec{a}{%
  +{}{\PrintAuthors} {author}
  +{,}{ \textit} {title}
  +{}{ \parenthesize} {date}
  +{,}{ } {note}
  +{.}{ } {transition}
}

\begin{document}

\title[Approximate transitivity of the ergodic  action of the group of finite permutations]{Approximate transitivity of the ergodic  action of the group of finite permutations of $\mathbb{N}$ on $\{0, 1 \}^{\mathbb{N}}$}

\author{B. Mitchell Baker}
\address[B. M. Baker]{ Mathematics Department, US Naval Academy, Chauvenet Hall,
572C Holloway Road,
Annapolis, MD 21402-5002}
\email{bmb@usna.edu}
\author{Thierry Giordano}
\address[T. Giordano]{ Department of Mathematics and Statistics, University of Ottawa, Ottawa,  K1N 6N5, Canada}
\email{giordano@uottawa.ca}
\author{Radu B. Munteanu}
\address[R. B. Munteanu]{Department of Mathematics, University of Bucharest, 14 Academiei St., 010014,  
Simion Stoilow Institute of Mathematics of the  Romanian Academy, 21 Calea Grivitei Street,  010702, Bucharest,  Romania}
\email{radu-bogdan.munteanu@g.unibuc.ro}

\begin{abstract}
In this paper we show that the natural action of the symmetric
group acting on the product space $\{0, 1 \}^{\mathbb{N}} $ endowed with a symmetric measure is approximately transitive. We also extend the result to a larger class of probability measures.
\end{abstract}
\maketitle

\section*{Introduction}

In 1985, Connes and Woods  introduced in \cite{CW1} the notion of an approximately transitive (AT) action, a new ergodic property, to characterize among approximately finite dimensional (AFD) von Neumann algebras, the Araki-Woods (or ITPFI) factors. Equivalently, using Krieger's result from \cite{K} (see \cite{HO} or \cite{S} for a detailed description), their result says that a countable, ergodic, non-singular equivalence relation on a Lebesgue space is orbit equivalent to the ergodic equivalence relation induced by a product odometer if and only if its associated flow is AT.
In 1989, for a locally compact group $G$, Connes and Woods \cite{CW2} proved that the asymptotic boundary of a group invariant, time-dependent Markov random walk on $G$ is an approximately transitive, amenable $G$-space. The converse statement was proved in two steps. First, that any amenable $G$-action can be realized as the asymptotic boundary of a generalized or matrix-valued random walk on $G$ was proved in \cite{EG1} in the discrete case and in \cite{AEG} for $G$ locally compact.

Then the characterization of an approximately transitive, amenable (ergodic) $G$-space as the asymptotic boundary of a random walk was given in \cite{EG2}; a different proof for $G$ discrete, was given in \cite{GH}.

In \cite{CW1} Connes and Woods proved that any funny rank one (a generalization of rank one) transformation is AT and that any AT transformation has zero entropy. Apart from some recent results (\cite{L}, \cite{DQ}) there are not many "concrete" examples of approximately transitive group actions.

Let $S_\infty=\bigcup_{k\geq 1} S_{k}$ be the group of finite permutations of $\mathbb{N}=\{1,2,3,\ldots \}$ and let $(X,\mathfrak{B},\nu)$ be the product space  $\prod_{k\geq 1}\{0,1\}$ endowed with the product $\sigma$-algebra and the product probability measure $\nu=\otimes_{k\geq 1}\nu_{k}$.

In this paper we concentrate our attention on the following well known action of $S_\infty$ on the product space $(X,\mathfrak{B},\nu)$ that associates to each permutation $\sigma\in\mathcal{S}_{\infty}$ a non-singular automorphism of $(X,\nu)$ (also denoted $\sigma$ and) defined by
\begin{equation*}\label{action}\tag{1}
\sigma(x_{1},x_2,x_3,\ldots)=(x_{\sigma(1)},x_{\sigma(2)},x_{\sigma(3)},\ldots) , \text{ for } x=(x_k)_{k\geq 1}\in X.
\end{equation*}

Before stating the main result of this paper, let us recall the definition of approximate transitivity.
\begin{definition*}\cite{CW1}
An action $\alpha$ of a Borel group $G$ on a Lebesgue measure space $(X,\nu)$ is
approximately transitive (abbreviated AT) if given $n<\infty$, functions $f_1,f_2, \ldots, f_n \in L^{1}_{+}(X,\nu)$ and $\varepsilon> 0,$
there exists a function $f\in L^{1}_{+}(X,\nu)$, elements $ g_1,\dots,g_m \ \in\ G$ for some
$m<\infty,$ and $\lambda_{j,k}\geq 0$, for $k=1,2,\dots,m$ and $j=1,2,\ldots,n$ such that
$$ \| f_j\ -\ \sum_{k=1}^m \lambda_{j,k}\beta_{g_{k}}(f) \| \leq\varepsilon,\quad
    j=1,\dots,n,$$
where $\|\ \cdot\ \|$ represents the $L^{1}$ norm and
$\beta_{g_k}(f) (x)\ = \ f\circ\alpha_{g_k^{-1}}(x)\frac{d\mu\circ g_k^{-1}}{d\mu}(x).$
\end{definition*}
Recall that Bernoulli measures on $X=\prod_{k\geq 1}\{0,1\}$ are the product measures $\nu_\lambda=\otimes_{k\geq 1} \nu_{\lambda,k}$ with $\nu_{\lambda,k}(0)=\frac{1}{1+\lambda}$ and $\nu_{\lambda,k}(1)=\frac{\lambda}{1+\lambda}$, where $0<\lambda\leq 1$.

Our main result proved in Section 2 is the following.

\begin{theorem*}
For $0<\lambda\leq 1$ the natural action of $S_{\infty}$ on $(X,\nu_{\lambda})$ is approximately transitive (AT).
\end{theorem*}

In Section 3, we generalize our main result to a larger class of product probability measures and show that the corresponding associated flow of $S_\infty$ is AT.

\section{Preliminaries}

Throughout this section, $(X,\nu)$ will denote the Lebesgue space $X=\prod_{k\geq 1}\{0,1\}$ and $\nu$ the product measure $\nu=\otimes_{k\geq 1}\nu_k$, with $\nu_k(0)= \frac 1{1+\lambda_k}$ and  $\nu_k(1)=\frac {\lambda_k}{1+\lambda_k}$, $0<\lambda_k\leq 1$.
In this section we prove some technical results, we will need in Section 2.
\begin{lemma}\label{l1}
Let $(X,\nu)$ be as above. For $0\leq r\leq n$, let $A(n,r)=\{x\in X: \ \#\{1\leq k\leq n: \  x_k=1\}=r \}$ denote the union of cylinder sets on $n$ symbols with exactly $r \quad 1'$s. Then
$$\nu (A(n,r))\quad<\quad \left(\frac{\pi}{\sum_{k=1}^n \lambda_k} \right)^{\frac1{2}},\quad 0\leq r \leq n. $$
\end{lemma}

\begin{proof} Let $$P_n(t)\ =\ \prod\limits_{k=1}^n\left(\frac{1}{1+\lambda_k}\quad + \quad
	\frac{\lambda_k}{1+\lambda_k}e^{it}\right).$$ Then it is easy to check that
	$$\nu (A(n,r))\ =\ \frac1{2\pi} \int_0^{2\pi} P_n(t)e^{-irt}\,dt\ .$$
As $0 < \lambda_k \leq 1$, we have	
\begin{align*}
	\nu (A(n,r))\quad &\leq \quad \frac1{2\pi}\int_0^{2\pi} |P_n(t)|\,dt \\
	  	   &=\quad \frac1{2\pi} \int_0^{2\pi} \prod\limits_{k=1}^n \left( 1\ -
			\frac{2\lambda_k}{(1+\lambda_k)^2}(1-\cos t) \right )^{\frac1{2}}\,dt\\
		   &\leq \frac1{2\pi} \int_0^{2\pi}\left\{ \prod\limits_{k=1}^n \left(1 \ - \
                 \frac{(1-\cos t)}{2} \lambda_k \right)\right\}^{\frac1{2}}\,dt,
	  \end{align*}
Now note that \ $\sqrt{1-x} \ \leq 1 - \frac{x}{2}\ $
and $\ \log (1-x)\leq -x\ $ for \ $0 < x < 1 \ $ and so
\begin{align*}
\nu(A(n,r))\quad &\leq \frac1{2\pi}\int_0^{2\pi}\exp\left(\frac{(\cos t-1)}{4}
	\sum_{k=1}^{n}\lambda_k\right)\,dt \quad \\&=\frac1{\pi}\int_0^{\pi}\exp\left(
	\frac{(\cos t-1)}{4}\sum_{k=1}^n\lambda_k\right)dt
\end{align*}
as $\ \cos t\ -1 $ is symmetric about $t\ =\ \pi$. Since $-t^2/\pi^2\ \geq \ \cos t - 1$
for $0\ \leq \ t \ \leq\pi,$ we have
\begin{align*}
\nu(A(n,r))\quad&\leq\quad\frac1{\pi}\int_0^{\pi}
e^{-\left(\sum_{k=1}^n\lambda_k\right)t^2/4\pi^2}\,dt <\quad\frac1{\pi}\int_0^{\infty}
e^{-\left(\sum_{k=1}^n\lambda_k\right)t^2/4\pi^2}\,dt  \\ & =\left(\frac{\pi}{\sum_{k=1}^n\lambda_k}\right)^{\frac 12}.
\end{align*}
\end{proof}
	
\begin{lemma}\label{l2}
Let $(X,\nu)$ be as above and $p$ be a fixed positive integer. Then, for $n\geq p$,
$$\sum_{r=p}^n | \nu(A(n,r)) - \nu(A(n,r-p)) |\quad < \quad 2p
	\left( \frac{\pi}{\sum_{k=1}^n\lambda_k} \right )^{\frac12}.$$
\end{lemma}

\begin{proof} Let
 $$P_n(t)\ = \ \prod\limits_{k=1}^n \left( \frac{1}{1+\lambda_{k}} + \frac{\lambda_k}{1+\lambda_k}e^{it}\right)
   \ = \ \sum_{k=0}^n \alpha_k e^{ikt}$$
with $\alpha_k\ =\ \nu(A(n,k))$. Then, for $p\leq n$,
\begin{align*}
	\|(1-e^{ipt})P_n\|_1\quad\ & =\quad\ |\alpha_0|+|\alpha_1|+ \dots + |\alpha_{p-1}|\\
						&+ \ \sum_{r=p}^n |\alpha_r - \alpha_{r-p} | +|\alpha_{n-p+1}| + \dots + |\alpha_{n}|,
  \end{align*}
where $\ \| \cdot \|_1 \ $ denotes the Fourier one-norm, i.e. the sum of the absolute values
of the Fourier coefficients. Thus
\begin{align*}
\sum_{r=p}^n |\nu(A(n,r))&-\nu(A(n,r-p))| = \sum_{r=p}^n |\alpha_r - \alpha_{r-p}| \\
  & \leq \  \|(1-e^{ipt})P_n\|_1  = \ \|\sum_{m=0}^{p-1}e^{imt}(1-e^{it})P_n\|_1 \\ & \leq \
    \sum_{m=0}^{p-1}\|e^{imt}(1-e^{it})P_n\|_1 =  p\|(1-e^{it})P_n\|_1 .
  \end{align*}
Now by Lemma 3.18 and Proposition 3.19 from \cite{B} we have
$$ \|(1-e^{it})P_n\|_1 \quad = \quad 2\max_{0 \leq r \leq n}\alpha_r \quad = \quad
     2\max_{0 \leq r \leq n}\nu(A(n,r)),$$
and by the above lemma, we conclude
$$\sum_{r=p}^n | \nu(A(n,r)) - \nu(A(n,r-p)) |\quad < \quad 2p
	\left( \frac{\pi}{\sum_{k=1}^n \lambda_k} \right )^{\frac12}.$$ The lemma is proved.
\end{proof}
\begin{remark}
Lemmas \ref{l1} and \ref{l2} are certainly well known to probabilists.
\end{remark}
For $\lambda\in (0,1]$ and the measure $\nu_\lambda=\otimes \nu_{\lambda, k}$ (with $\nu_{\lambda, k}(0)=\frac{1}{1+\lambda}$, $\nu_{\lambda, k}(1)=\frac{\lambda}{1+\lambda}$), we have
$$\nu_\lambda\left(A(n,r)\right)= \binom{n}{r}\frac{\lambda^r}{(1+\lambda)^n}.$$
Hence Lemma \ref{l1} and \ref{l2} become
\begin{corollary}\label{cor}
Let $p$ be a positive integer and $\lambda\in(0,1]$. Then
\begin{align*}
&1) \ \binom{p}{r}\frac{\lambda^{r}}{(1+\lambda)^{p}}<\left(\frac{\pi(1+\lambda)}{p\lambda}\right)^{\frac{1}{2}}\text{ for every }0\leq r\leq p,\\
&2) \ \sum_{r=n}^{p}\left| \binom{p}{r}
		\frac{\lambda^{r}}{(1+\lambda)^{p}}
		-  \binom{p}{r-n}\frac{\lambda^{r-n}}{(1+\lambda)^{p}}\right|< 2n
			\left( \frac{\pi(1+\lambda)}{p\lambda} \right )^{\frac12}\text{ for every }0<n\leq p.
\end{align*}			
		
			\end{corollary}

\section{Main Result}

Let $0<\lambda\leq 1$ and $\nu_\lambda=\otimes_{k\geq 1}\nu_{\lambda,k}$ be a Bernoulli measure on $X=\prod\limits_{k=1}^{\infty}  \{0  , 1\}$. Recall (see for example \cite{AP} or \cite{SV}) that the natural action of $S_\infty$ on $(X,\nu_\lambda)$ is then ergodic and measure preserving. Moreover any $S_\infty$-invariant, ergodic probability measure on $X$ is a Bernoulli measure. In this section, we prove the theorem stated in the introduction:

\begin{theorem}\label{AT}
The natural action of $S_{\infty}$ on $(X,\nu_\lambda)$ is approximately transitive.
\end{theorem}

Let us introduce notations, we will need below. For $n\geq 1$, let $X^{n}=\prod\limits \limits_{k=1}^{n}\{0,1\}$. For $x=(x_1,x_2,\ldots, x_n)\in X^{n}$ let $$C(x)=\{y\in X: y_{i}=x_{i}, i=1,2,\ldots,n\}$$ denote the cylinder set of length $n$ whose first $n$ symbols are given by $x$. 
Let $$\mathscr{S}_{n}^+=\left\{ \sum_{x\in X^{n}}\alpha_{x}\chi_{C(x)}, :  \alpha_{x}\in\mathbb{R}_+ \right\}.$$ Note that for $m\geq n$, every cylinder set of length $n$ decomposes into a disjoint union of cylinders of length $m$, and so,  $\mathscr{S}_{n}^+\subseteq \mathscr{S}_{m}^+$.

The proof of Theorem \ref{AT} will follow from the two technical lemmas bellow. The first one is well known and we omit its proof.

\begin{lemma}\label{l3}
Let $f\in L^{1}_{+}(X,\nu)$ be a non-negative real valued function. Then for every $\varepsilon>0$, there exist $n\geq 1$ and $g\in \mathscr{S}_{n}^{+}$ such that  $\|f-g\|\leq \varepsilon$.
\end{lemma}

\begin{lemma}\label{23}
For each $n\geq 1$ and $0<\varepsilon\leq 1$, there exist a positive integer $p$, a function $f\in  \mathscr{S}_{p}^{+}$, a finite subset $S\subset S_{\infty}$ and nonnegative constants $a_{x,\sigma}$, for $x\in X^n$ and $\sigma\in S$ such that for all $x\in X^n$,
$$ \|\chi_{C(x)} - \sum_{\sigma\in S}a_{x,\sigma} \beta_{\sigma}(f) \| \ \leq \ \varepsilon\|\chi_{C(x)}\|.$$
\end{lemma}
\begin{proof}  For $x\ =\ (x_1,\dots ,x_n)\ \in\  X^n$ define $\#(x)$ to be the number of $1$'s in the n-tuple. Let $0<\varepsilon\leq 1$ be given and choose a positive integer $p$ such that
$$p\geq \max\left\{ 2n, \frac{\pi(1+\lambda)}{\lambda}\cdot \frac{18n^2}{\varepsilon^2} \right\}.$$

For $\iota\in\{0,1\}$, to simplify the notation, we will write $\iota_l$ for  a string of $l$ $\iota$'s.
For $n\leq j\leq p-n$, let $C(p,j)$ denote the cylinder set $C(1_j, 0_{p-j})$ and define $f\in \mathscr{S}_p^+$ by
\begin{equation}\label{222}
f\quad = \quad \sum_{j=n}^{p-n}\frac{\binom{p-n}j}{(p-n)!}\ \chi_{C(p,j)} \ .
\end{equation}
For each $x\in X^{k_n}$, let $x^{-1}(0)$ denote the set $\{j\in\{1,2,\ldots, k_n\}: x_j=0\}$ and consider the permutation $\sigma_x$ given by
$$\sigma_x=\prod_{j\in x^{-1}(0)}(j, p-n+j)$$
where $(a,b)$ denotes the transposition between $a$ and $b$.

By construction, for $n\leq j\leq p-n$, we have
\begin{equation}\label{223}
\sigma_x(C(p,j))=C(x,1_{j-n}, 0_{p-n-j}, \overline{x})
\end{equation}
where if $x=(x_1,x_2,\ldots ,x_n)$, then $\overline{x}=(1-x_1,1-x_2,\ldots ,1-x_n)$
(the subtraction is computed modulo $2$).

Recall that for $0\leq j\leq p$, $A(p,j)$ denotes the union of all cylinders on $p$ symbols with exactly $j$ $1'$s. If
$$S_{n,p} = \{\tau \in S_\infty \ : \ \tau(i)= i \ , \ 1 \leq i \leq n \text{ or }i\geq p+1\},$$
then from (\ref{223}) we get for $n\leq j\leq p-n$
$$\bigcup_{\tau\in S_{n,p}}\tau\sigma_x(C(p,j))=C(x)\cap A(p,j).$$
Moreover, the stabilizer in $S_{n,p}$ of $(x,1_{j-n}, 0_{p-n-j},\overline{x})$ is isomorphic to the product of the symmetric groups $S_{j-k}$ and $S_{p-n-j+k}$, where $k=\#(x)$. Then
$$\sum_{\tau\in S_{n,p}}\chi_{\tau\sigma_x(C(p,j))}=(j-k)!(p-n-j+k)!\chi_{C(x)\cap A(p,j)}=
\frac{(p-n)!}{\binom{p-n}{j-k}}\chi_{C(x)\cap A(p,j)}.$$
By definition (\ref{222}) of $f$, we then have
\begin{equation}\label{224}
\sum_{\tau\in S_{n,p}}\beta_{\tau \sigma_x}(f)=\sum_{j=n}^{p-n}\frac{\binom{p-n}{j}}{\binom{p-n}{j-k}}\chi_{C(x)\cap A(p,j)}.
\end{equation}
Note that  for $k\leq j\leq p-(n-k)$,
\begin{equation}\label{225}
\left\| \chi_{C(x)\cap A(p,j)} \right\|= \left\| \chi_{C(x)} \right\|\binom{p-n}{j-k}\frac{\lambda^{j-k}}{(1+\lambda)^{p-n}}
\end{equation}
Let $B(p,n)=\bigcup\limits_{j=n}^{p-n}A(p,j)$.
 By Corollary \ref{cor}, we have
\begin{align}\label{221}
\nu_\lambda( B(p,n)^c &\cap C(x))=\nu_\lambda(\bigcup_{j=k}^{n-1} (C(x)\cap A(p,j)))+\nu_\lambda\left(\bigcup_{j=p-n+1}^{j=p-(n-k)} (C(x)\cap A(p,j))\right)  \nonumber \\ &=\left[\sum_{j=k}^{n-1}\binom{p-n}{j-k}\frac{\lambda^{j-k}}{(1+\lambda)^{p-n}} +\sum_{j=p-n+1}^{p-(n-k)}\binom{p-n}{j-k}\frac{\lambda^{j-k}}{(1+\lambda)^{p-n}}\right]\|\chi_{C(x)}\| \nonumber \\
&\leq n\left( \frac{\pi(1+\lambda)}{(p-n)\lambda} \right)^{\frac{1}{2}}\|\chi_{C(x)}\|.
\end{align}
From (\ref{224}), (\ref{225}), Corollary \ref{cor} and as $k\leq n$, 
	\begin{align*}
			\| \chi_{C(x)\cap B(n,p)} & -  \lambda^{k}\sum_{\tau \in S_{n,p}} \beta_{\tau \sigma_x}(f) \|
			  \quad \leq \quad \left\| \sum_{j=n}^{p-n} \left(1 - \lambda^{k} \cdot \frac{\binom{p-n}{j}}
				{\binom{p-n}{j-k}}\right) \chi_{C(x)\cap A(n,j)} \right\| \\
			&\leq  \|\chi_{C(x)}\|  \sum_{j=n}^{p-n}\left| \binom{p-n}{j-k}\frac{\lambda^{j-k}}{(1+\lambda)^{p-n}}\quad - \quad \binom{p-n}{j} \frac{\lambda^{j}}{(1+\lambda)^{p-n}}
				\right| \\
&\leq  \|\chi_{C(x)}\|  \sum_{j=k}^{p-n}\left| \binom{p-n}{j-k}\frac{\lambda^{j-k}}{(1+\lambda)^{p-n}}\quad - \quad \binom{p-n}{j} \frac{\lambda^{j}}{(1+\lambda)^{p-n}}
				\right| \\	
				& \leq  2k\|\chi_{C(x)}\|\left(\frac {\pi (1+\lambda)}{(p-n)\lambda} \right)^{\frac{1}{2}} \leq  2n\|\chi_{C(x)}\|\left(\frac {\pi (1+\lambda)}{(p-n)\lambda} \right)^{\frac{1}{2}}
\end{align*}
By (\ref{221}) and the choice of $p$, we get
\begin{align*}
			\| \chi_{C(x)} - & \lambda^{k}\sum_{\tau \in S_{n,p}} \beta_{\tau \sigma_x}(f) \| \\
			 & \quad \leq 	\left\| \chi_{C(x) \cap B(n,p)^c}\right\|+\left\| \chi_{C(x) \cap B(n,p)}-\lambda^{k}\sum_{\tau \in S_{n,p}} \beta_{\tau \sigma_x}(f) \right\| \\
			  & \quad  \leq  \varepsilon \|\chi_{C(x)}\|		
			  \end{align*}
The proof of the lemma is completed by letting $S$ be the (disjoint) union of $S_{n,p}\sigma_x$, $x\in X^n$ and
$$a_{x,\sigma}=\left\{ \begin{array}{cc}
\lambda^{\# (x)} & \text{ if }\sigma\in S_{n,p}\sigma_x\\
0  & \text{ if not.\ \ \ \ \ \ \ \ \ }
\end{array}
\right.$$
\end{proof}

\bigskip

\begin{proof}[Proof of Theorem \ref{AT}] 

Consider an arbitrary finite collection $f_1, f_{2},\ldots, f_{m}$ of functions from $L^{1}_{+}(X,\nu_\lambda)$ and $\varepsilon>0$. Choose $\eta>0$ such that $(1+\|f_j\|)\cdot\eta\leq \varepsilon$ for all $j$. By Lemma \ref{l3}, there exists $n\geq 1$ and  nonnegative coefficients $\lambda_{j,x}$ such that
\[\|f_{j}\|=\sum_{x\in X^{n}}\lambda_{j,x}\|\chi_{C(x)}\| 
  \text{ \ \  and \ \ } \|f_{j}-\sum_{x\in X^{n}}\lambda_{j,x}\chi_{C(x)}\|\leq\eta, \quad  j=1,2,\ldots,m.\]
By Lemma \ref{23} there exists $f\in L^{1}_{+}(X,\nu_\lambda)$,  a finite set $\mathcal{S}$ of elements of
$S_{\infty}$ and reals $a_{x,\sigma}\geq 0$, $x\in X^{n}$, $\sigma\in \mathcal{S}$ which satisfy
$$ \|\chi_{C(x)} - \sum_{\sigma\in S}a_{x,\sigma} \beta_{\sigma}(f) \| \ \leq \ \eta\|\chi_{C(x)}\|\quad
		\text{for all} \ x \in X^{n}.$$
Then, for $j=1,2,\ldots, n$ we obtain
		\begin{align*}
\|f_j-&\sum_{x\in X^{n}}\sum_{\sigma\in\mathcal{S}}\lambda_{j,x}a_{x,\sigma}\beta_{\sigma}(f)\|\\&\leq\|f_{j}-\sum_{x\in X^{n}}\lambda_{j,x}\chi_{C(x)}\|+ \sum_{x\in X^{n}}\lambda_{j,x} \|\chi_{C(x)} - \sum_{\sigma\in\mathcal{S}} a_{x,\sigma}\beta_{\sigma}(f) \|\\
&\leq \eta\cdot\left(1+\sum_{x\in X^{n}} \lambda_{j,x}\|\chi_{C(x)}\|\right)=\eta(1+\|f_j\|)\leq \varepsilon.
\end{align*}
Hence the action of $\mathcal{S}_{\infty}$ on $(X,\nu_\lambda)$ is AT.\end{proof}

\bigskip

\section{Generalizations}
In this section, we generalize Theorem \ref{AT} for a larger class of probability measures on $X=\prod_{k\geq 1}\{0,1\}$, and show that the associated flow of $S_\infty$ is AT for this class of measures.  Let $(L_n)_{n\geq 1}$ be a sequence of positive integers and $(\lambda_n)_{n\geq 1}$ be a sequence of real numbers in $(0,1]$. Then $\nu=\nu(L_n,\lambda_n)$ will denote the product measure on $X=\prod_{k\geq 1}\{0,1\}=\prod\limits_{n\geq 1}\prod\limits_{1}^{L_n}\{0,1\}$,
given by
$$\nu(L_n,\lambda_n)=\otimes_{n\geq 1}\nu_{\lambda_n}^{\otimes L_n},$$
where $\nu_{\lambda_n}(0)=\frac{1}{1+\lambda_{n}}$ and $ \nu_{\lambda_n}(1)=\frac{\lambda_{n}}{1+\lambda_{n}}$.
In this section we will assume that
\begin{equation}\label{cond}
\sup\{L_{n}\lambda_{n} : n\geq 1\}=\infty.
\end{equation}

Recall (see \cite{SV} or \cite{AP}) that $\nu(L_n,\lambda_n)$ is a $S_\infty$-ergodic, nonatomic measure on $X$ if and only if $\sum_{n\geq 1}L_n\lambda_n=\infty$.

The following lemma and its proof are a generalization of Lemma \ref{23}. We will use the following notations: set $k_0=0$ and for $n\geq 1$, $k_n=\sum\limits_{k=1}^n L_k$ and recalling that for $n\geq 1$, $X^n=\prod\limits_{k=1}^n\{0,1\}$, we set for $x\in X^{k_n}$ and $z\in X^{L_l}$, $l\geq 1$,
$$C(x)=\{y\in X :  y_j=x_j\text{ for }1\leq j\leq k_n\}$$
$$C(z(L_l))=\{y\in X: y_{k_{l-1}+i}=z_i \text{ for }1\leq i\leq L_l \} $$
and for $m>n$
$$C(x, z(L_m))=C(x)\cap C(z(L_m)).$$
\begin{lemma}\label{33}
Let $(X,\nu)$ be as above. Then for each $n\geq 1$ and $\varepsilon>0$, there exist a function $f\in  L^1_+(X,\nu)$, a finite subset $S\subset S_{\infty}$ and nonnegative constants $a_{x,\sigma}$, for $x\in X^{k_n}$ and $\sigma\in S$ such that
$$ \|\chi_{C(x)} - \sum_{\sigma\in S}a_{x,\sigma} \beta_{\sigma}(f) \| \ \leq \ \varepsilon\|\chi_{C(x)}\|$$
for all $x\in X^{k_n} $.
\end{lemma}
\begin{proof} Let $n\geq 1$ and $0<\varepsilon< 1$. By (\ref{cond}) we can choose an integer $m\geq 1$ such that
\begin{equation}\label{30}
3k_n\left(\frac{\pi(1+\lambda_m)}{L_m\lambda_m} \right) ^{\frac{1}{2}}\leq\varepsilon.
\end{equation}
Notice that $k_n < L_m$. For $0\leq j\leq L_m-k_n$, and $\widetilde{z}_j=(1_j,0_{L_m-j})\in X^{L_m}$, set
\begin{equation*}\label{cyl1}
C(n,m,j)=C(1_{k_n})\cap C(\widetilde{z}_j(L_m))
\end{equation*}
and
\begin{equation*}
C(n,m)=\bigcup_{j=0}^{L_m-k_n}C(n,m,j).
\end{equation*}
We then define the function $f\in L^1_+(X,\nu)$ by
\begin{equation}\label{function}
f \quad = \quad \sum_{j=0}^{L_{m}-k_n}\frac{\binom{L_{m}}{k_n+j}}{L_m!}\ \chi_{C(n,m,j)},
\end{equation}
whose support is $C(n,m)$.
For an arbitrary $x\in X^{k_n}$, let $x^{-1}(0)$ denote the set $\{j\in\{1,2,\ldots, k_n\}: x_j=0\}$ and consider the permutation $\sigma_x$ given by
$$\sigma_x=\prod_{j\in x^{-1}(0)}(j, k_m-k_n+j)$$
where $(a,b)$ denotes the transposition between $a$ and $b$.
For all $0\leq j\leq L_m-k_n$, set
$$z_j=(1_j,0_{L_m-k_n-j},\overline{x}),$$
where $\overline{x}=(1-x_1,1-x_2,\ldots,1-x_{k_n})\in X^{k_n}$ (the subtraction is computed modulo $2$). Then by construction
\begin{equation}
\sigma_x(C(n,m,j))=  C(x, z_j(L_m))
\end{equation}
and for every $y\in C(x,z_j(L_m))$,
\begin{align*}
\frac{d\nu\circ \sigma_x^{-1}}{d\nu}(y) & =\quad \prod_{k\geq 1}\ \ \frac{\nu_k(\sigma_x^{-1}(y)_k)}{\nu_k(y_k)}\\ & =\prod_{j\in x^{-1}(0)}\frac{\nu_j(y_{k_m-k_n+j})}{\nu_j(y_j)}\frac{\nu_{k_m-k_n+j}(y_j)}{\nu_{k_m-k_n+j}(y_{k_m-k_n+j})}\\
 & =\prod_{j\in x^{-1}(0)}\frac{\nu_j(1)}{\nu_j(0)}\frac{\nu_{k_m-k_n+j}(0)}{\nu_{k_m-k_n+j}(1)}=
\prod_{j\in x^{-1}(0)}\frac{\nu_j(1)}{\nu_j(0)}\lambda_m^{-1}
\end{align*}
and therefore is constant on $\sigma_x(C(n,m))$; let $D_x$ be this constant.
By definition of $f$ we have
\begin{align}\label{301}
\beta_{\sigma_x}(f) = \sum_{j=0}^{L_{m}-k_n}\frac{\binom{L_m}{k_n+j}}{L_{m}!}
\chi_{C(x, z_j(L_m))}D_x.
\end{align}
Let $S(L_m)$ be the subgroup of $S_\infty$ of permutations $\sigma$ such that $\sigma(i)=i$, for $i\notin \{k_{m-1}+1,\ldots, k_m\}$. For $0\leq l\leq L_m$, set
\begin{align*}
A(x,m,l)=C(x)\bigcap \left\{y\in X: \text{card}\{i: k_{m-1}+1\leq i\leq k_m, \  y_{i}=1 \}=l\right\}.
\end{align*}
Then it is easy to observe that for $0\leq j\leq L_m-k_n$
$$\bigcup_{\tau\in S(L_m)}\tau (C(x,z_j(L_m))=A(x,m,k_n-k+j),$$
where $k=\sum\limits_{i=1}^{k_n}x_i=\text{card}\{i : \ 1\leq i\leq k_n, \  x_i=1 \}$.
As the stabilizer in $S(L_m)$ of $z_{j}$ is isomorphic to $S_{j+k_n-k}\times S_{L_m-k_n-j+k}$, we have for $0\leq j\leq L_m-k_n$
\begin{align*}
\sum_{\tau\in S(L_m)}\chi_{\tau(C(x,z_j(L_m)))} & =(k_n-k+j)!(L_m-k_n-j+k)!\chi_{A(x,m,k_n-k+j)}\\
& =\frac{L_m !}{\binom {L_m}{k_n-k+j}}\chi_{A(x,m,k_n-k+j)}.
\end{align*}
Any $\tau\in S(L_m)$ being $\nu$-measure preserving, we get
\begin{align}\label{333}
\sum_{\tau\in S(L_m)} \beta_{\tau\sigma_x}(f)\quad  =  \sum_{j=0}^{L_m-k_n}
D_x	\frac{\binom{L_m}{k_n+j}}{\binom{L_m}{k_n-k+j}} \chi_{A(x,m,k_n-k+j)}.
\end{align}
Note that  for $0\leq j\leq L_m$,
\begin{equation}\label{31}
\left\|\chi_{A(x,m,j)} \right\|= \left\| \chi_{C(x)} \right\|\binom{L_m}{j}\frac{\lambda_m^{j}}{(1+\lambda_m)^{L_m}}.
\end{equation}
Let $B(x,m)=\bigcup\limits_{j=k_n-k}^{L_m-k_n}A(x,m,j)=\bigcup\limits_{j=0}^{L_m-k_n}A(x,m,k_n-k+j)$. From (\ref{333}) and (\ref{31}), we get easily
\begin{align*}
\| \chi  _{C(x)\cap  B(x,m)} &  -  \frac{\lambda_{m}^{k}}{D_x} \sum_{\tau\in S(L_m)} \beta_{\tau \sigma_x} (f) \|  \\
& \leq\sum_{j=0}^{L_m-k_n}\left|\binom{L_{m}}{k_n-k+j}\frac{\lambda_{m}^{k_n-k+j}}{(1+\lambda_{m})^{L_{m}}} - \binom{L_m}{k_n+j}\frac{\lambda_{m}^{k_n+j}}{(1+\lambda_{m})^{L_{m}}}	 \right|\|\chi_{C(x)}\| \nonumber
\end{align*}
Then by Corollary \ref{cor},
\begin{align}\label{rel1}
	\| \chi_{C(x)\cap B(x,m)}  -  \frac{\lambda_{m}^{k}}{D_x} \sum_{\tau\in S(L_m)} \beta_{\tau \sigma_x} (f) \|
	\leq 2k_n\left(\frac{\pi(1+\lambda_m)}{L_m\lambda_m} \right) ^{\frac{1}{2}}\|\chi_{C(x)}\|
\end{align}
and
\begin{align}\label{rel2}
\| \chi_{C(x)\cap B(x,m)^c}\|=
& \  \sum_{j=0}^{k_n-k-1} \|\chi_{A(x,m,j)}\| + \sum_{j=L_m-k+1}^{L_m} \|\chi_{A(x,m,j)}\| \nonumber \\
& \leq k_n\left(\frac{\pi(1+\lambda_m)}{L_m\lambda_m} \right) ^{\frac{1}{2}}\|\chi_{C(x)}\|.
\end{align}
Therefore, by (\ref{30}), (\ref{rel1}) and (\ref{rel2}), we get
\begin{align*}
\| \chi_{C(x)}  -  \frac{\lambda_{m}^{k}}{D_x} \sum_{\tau\in S(L_m)} \beta_{ \tau\sigma_x} (f) \|  \leq   3k_n\left(\frac{\pi(1+\lambda_m)}{L_m\lambda_m} \right) ^{\frac{1}{2}}\|\chi_{C(x)}\|\leq \varepsilon \|\chi_{C(x)}\|.
\end{align*}

We let $S$ be the (disjoint) union of $S(L_m)\sigma_x$, $x\in X^{n}$ and we define
$$a_{x,\sigma}=\left\{ \begin{array}{cc}
{\lambda^{k}_m}/D_x & \text{ if }\sigma=\tau\sigma_x, \  \tau\in S(L_m)\\
0  & \text{ otherwise.\ \ \ \ \ \ \ \ \ }
\end{array}
\right.$$
\end{proof}

\noindent By Lemma \ref{33} and Lemma \ref{l3}, we then get:

\begin{theorem}
Let $\nu=\nu(L_n,\lambda_n)$ be a product probability measure on $X$ satisfying
\begin{equation*}
\sup\{L_{n}\lambda_{n} : n\geq 1\}=\infty.
\end{equation*} Then the natural action of $S_\infty$ on $(X,\nu)$ is AT.
\end{theorem}

\begin{remark}\label{remk}
Keeping the notation of Lemma \ref{33} and its proof, note that for any $\sigma\in S\subset S_\infty$, the Radon-Nikodym derivative $\frac{d\nu\circ \sigma^{-1}}{d\nu}$ is constant on $\sigma\left(C(n,m)\right)$ where $C(n,m)$ is the support of $f$. Indeed, $S$ is the disjoint union of $S(L_m)\sigma_x$, for $x\in X^n$, moreover any $\tau\in S(m)$ is $\nu$-measure preserving and
$$\frac{d\nu\circ\sigma_{x}^{-1}}{d\nu}(y)=D_x=\prod_{j\in x^{-1}(0)}\frac{\nu_i(1)}{\nu_i(0)}\lambda_m^{-1} \text{ for }  y\in \sigma_x(C(n,m)).$$
\end{remark}
\bigskip

Recall that the associated flow of a nonsingular action of a countable discrete group $G$ on a Lebesgue space $(X,\nu)$ is the Mackey range of the Radon-Nikodym cocycle of the $G$-action on $(X,\nu)$. More precisely, on $(X\times\mathbb{R},\ \nu\otimes e^{u}du)$, let us denote by $\gamma$ the action of $G$ and by $\rho$ the action of $\mathbb{R}$, given by
\begin{align*}
\gamma_{g}(x,t)&=\left(gx,t-\log\frac{d\mu\circ g}{d\mu}(x)\right)
\end{align*}
and
\begin{align*}
\rho_s(x,t)&=(x,t+s).
\end{align*}
As $\rho$ commutes with the $G$-action $\gamma$, it induces an $\mathbb{R}$-action on the ergodic decomposition of the (infinite measure preserving) action $\gamma$, which is the associated flow of the action of $G$ on $(X,\nu)$.
Let $$\widetilde{\gamma}_{g,s}(x,t)=\left(gx,t+s-\log\frac{d\mu\circ g}{d\mu}(x)\right)$$
be the skew product action $\widetilde{\gamma}$ of $G\times\mathbb{R}$ on $(X\times \mathbb{R}, \nu\otimes \mu)$ where $d\mu=e^udu$.

As by \cite{CW1}, Remark 2.4, any factor action of an AT action is AT, to prove that the associated flow of a $G$-action on $(X,\nu)$ is AT it is sufficient to show that the above skew product action $\widetilde{\gamma}$ is AT.

We introduce in Definition \ref{atilde} a strong version of AT, which we denote by $\widetilde{AT}$, and show in Proposition 3.5 that if a nonsingular $G$-action is $\widetilde{AT}$, its associated skew product $(G\times\mathbb{R})$-action is AT.
\begin{definition}\label{atilde}
Let $(X,\nu,G)$ be a nonsingular action of a countable discrete group $G$ on a Lebesgue space $(X,\nu)$. Then the action is $\widetilde{AT}$ if given $n<\infty$, functions $f_1,f_2, \ldots, f_n \in L^{1}_{+}(X,\nu)$ and $\varepsilon> 0,$
there exists a function $f\in L^{1}_{+}(X,\nu)$, elements $ g_1,\dots,g_m \ \in\ G$, $s_1,\ldots, s_m>0$ for some
$m<\infty,$ and $\lambda_{j,k}\geq 0$, for $k=1,2,\dots,m$ and $j=1,2,\ldots,n$ such that
$$ \| f_j\ -\ \sum_{k=1}^m \lambda_{j,k}\beta_{g_{k}}(f) \| \leq\varepsilon,\quad  j=1,\dots,n$$
and
$$\frac{d\nu \circ g_k^{-1}}{d\nu}(x)=s_k, \ \nu \text{ a.e. }x\in g_k\left(\text{Supp}(f)\right),\quad  k=1,\ldots, m,$$
where $\text{Supp}(f)$ denotes the support of $f$.
\end{definition}
\begin{proposition}\label{skew}
Let $(X,\nu, G)$ be a nonsingular  $\widetilde{AT}$ action of a countable group $G$ on a Lebesgue space $(X,\nu)$. Let $\mu$ denote the measure on $\mathbb{R}$ given by $d\mu=e^{u}du$. Then the skew product action $\widetilde{\gamma}$ of $G\times\mathbb{R}$ on $(X\times \mathbb{R}, \nu\times \mu)$ is AT.
\end{proposition}
\begin{proof} Let $0<\varepsilon<1$ and consider $n$ nonnegative functions $\widetilde{f}_j\in L^{1}_+(X\times\mathbb{R},  \nu\otimes\mu)$,  $j=1,2,\ldots, n$. By standard approximation arguments, we can assume that $\widetilde{f}_j=f_j\otimes f'_j$, where $f_j\in L^1_+(X,\nu)$ and $f'_j\in  L^1_+(\mathbb{R},\mu)$. By assumption, there exist $f\in  L^1_+(X,\nu)$, $g_1,\ldots, g_m\in G$, $s_1,\ldots, s_m\in \mathbb{R}_+^*$ for some $m<\infty$, and $\lambda_{j,k}\in \mathbb{R}_+$, for $k=1,\ldots, m$ and  $j=1,\ldots, n$ such that
\begin{align*}
\|f_j-\sum_{k=1}^m \lambda_{j,k}\beta _{g_k}(f)\|\leq \varepsilon\|f_j\|,  \  j=1,\ldots,n
\end{align*}
and
\begin{align*}
\frac{d\nu \circ g_k^{-1}}{d\nu}(x)=s_k, \ \nu \text{ - a.e. }x\in g_k\left(\text{Supp}(f)\right), \  k=1,\ldots, m.
\end{align*}
As the action $\rho$ by translation on $(\mathbb{R}, \mu)$ is transitive and therefore AT, there exists $f'\in L^{1}_+(\mathbb{R}, \mu)$, $t_1,\ldots, t_p\in\mathbb{R}$ for some $p<\infty$ and $\lambda'_{j,l}\geq 0$, $l=1,\ldots, p$, $j=1,\ldots n$ such that
$$\|f'_j-\sum_{l=1}^p \lambda'_{j,l}\rho _{t_l}(f')\|\leq \varepsilon \|f'_j\|$$
where
$$\rho_s(f')(t)=f'(t-s)\frac{d\mu\circ\rho_{-s}}{d\mu}(t)=e^{-s}f'(t-s),  t\in\mathbb{R}.$$
Note first that
$$\|f_j\otimes f'_j-\sum_{k=1}^m \lambda_{j,k}\beta _{g_k}(f)\otimes f'_j\|\leq \varepsilon\|f_j\|\|f'_j\|=\varepsilon\|f_j\otimes f'_j\|$$
and that
$$\|\sum_{k=1}^m \lambda_{j,k}\beta _{g_k}(f)\otimes  f'_j-\sum_{k=1}^m \sum_{j=1}^p \lambda_{j,k}\lambda'_{j,l}\beta _{g_k}(f)\otimes \rho _{t_l}(f')\|\leq 2\varepsilon\|f_j\|\|f'_j\|=2\varepsilon\|f_j\otimes f'_j\|.$$
By definition of the $(G\times \mathbb{R})$-action $\widetilde{\gamma}$ and as for
$1\leq k\leq m$, $\frac{d\nu\circ g_k^{-1}}{d\nu}(x)=s_k$  for $\nu$-a.e. $x$ with  $g_k^{-1}x\in\text{Supp}(f)$,
\begin{align*}
\widetilde{\gamma}_{g_k,t_l-\log s_k}(f\otimes f')(x,t)&=f(g_k^{-1}x)f'(t-\log \frac{d\nu\circ g_k^{-1}}{d\nu}(x)-t_l+\log s_k) e^{-t_l+\log s_k}\\&=\beta_{g_k}(f)(x)\rho_{t_l}(f')(t), \ \text{ for }  (\nu\otimes \mu) \text{ - a.e. }(x,t)\in X\times\mathbb{R}.
\end{align*}
Hence,
$$\|f_j\otimes f'_j- \sum_{k=1}^m \sum_{j=1}^p \lambda_{j,k}\lambda'_{j,l}\widetilde{\gamma}_{g_k,t_l+\log s_k}(f\otimes f')\|\leq 3\varepsilon\|f_j\|\|f'_j\|=3\varepsilon\|f_j\otimes f'_j\|$$
and therefore the skew action $\tilde{\gamma}$ is AT. \end{proof}
\bigskip

By Lemma \ref{33} and Remark \ref{remk} if $\nu_n=\nu(L_n, \lambda_n)$ is a product probability measure on $X=\prod_{k\geq 1}\{0,1\}$ satisfying (\ref{cond}), then the natural action of $S_\infty$ on $(X,\nu)$ is $\widetilde{AT}$.

By \cite{CW1}, Remark 2.4 and Proposition \ref{skew} we then get:
\begin{theorem}\label{the}
Let $\nu=\nu(L_n, \lambda_n)$ be a product probability measure on $X=\prod_{k\geq 1}\{0,1\}$ satisfying (\ref{cond}). Then the associated flow of the natural action of $S_\infty$ on $(X,\nu)$ is AT.
\end{theorem}
\noindent The previous results can be generalized as follows:
\begin{proposition}\label{pro}
Let $\nu=\otimes_{k\geq 1}\nu_{k}$ be a product probability measure on $X=\prod_{k\geq 1}\{0,1\}$. We assume that there exists real numbers $\lambda_n\in (0,1]$, $n\geq 1$ and mutually disjoint sets of positive integers $J_n$, $n\geq 1$ of cardinality $L_n$ such that
$\nu_k=\nu_{\lambda_n}, \text{ for }k\in J_n$
and such that at least one of the $J_n$'s is infinite or
\begin{equation*}
\sup \ \{ L_{n}\lambda_{n} :  n\geq 1\}=\infty
\end{equation*}
if all $J_n$'s are finite.
Then the natural action of $S_\infty$ on $(X,\nu)$ is $\widetilde{AT}$ and the corresponding skew product action is AT.
\end{proposition}
\begin{corollary}\label{c1}
Let $\nu=\otimes_{k\geq 1}\nu_{\lambda_k}$ be a product probability measure on $X=\prod\limits_{k\geq 1}\{0,1\}$. If the sequence $(\lambda_{k})_{k\geq 1}$ has a non zero limit point $\lambda$, then the action of $S_{\infty}$  on  $(X,\nu)$ and the action $\widetilde{\gamma}$ of $S_{\infty}\times \mathbb{R}$ on  $(X\times \mathbb{R},\nu\otimes e^udu)$ are AT.
\end{corollary}

\begin{proof} Since $(\lambda_k)_{k\geq 1}$ has a nonzero limit point $\lambda$ we can choose an infinite set $J$ such that $\sum_{j\in J}(\lambda_j-\lambda)^2<\infty$.
For $k\in J$, define the measure $\mu_k$ on $\{0,1\}$, by $\mu_k(0)=\frac{1}{1+\lambda}$ and $\mu_k(1)=\frac{\lambda}{1+\lambda}$. For $j\notin J$, let $\mu_k=\nu_k$. Consider the product measure $\mu=\otimes_{k\geq 1}\mu_k$. By Kakutani's theorem (see for example \cite{HS}, Theorem 22.36), the measures $\mu$ and $\nu$ are equivalent.
Proposition \ref{pro} implies that both the action of $S_\infty$ on $(X,\mu)$ and the action $\widetilde{\gamma}$ of $S_{\infty}\times \mathbb{R}$ on $(X\times \mathbb{R},\mu\otimes e^udu)$ are AT. Since AT property is preserved when the measure is replaced with an equivalent one, the corollary follows.
\end{proof}

\begin{remark}
In general, it is not known if for any product probability measure $\nu$ on $X=\prod_{k\geq 1}\{0,1\}$, the natural action of $S_\infty$ on $(X,\nu)$, or its associated flow are AT.
\end{remark}

\begin{remark}
On the product space $X=\prod_{k\geq 1}\{0,1\}$ consider the product probability measure $\nu=\otimes_{k\geq 1}\nu_{\lambda_k}$, where $\lambda_k=\frac{1}{k}$ for $k\geq 1$. We do not know if the system $(X,\nu, S_\infty)$ is AT. On the other hand, it can be shown that $(X,\nu, S_\infty)$ is of type III$_1$ and therefore its associated flow is AT.
\end{remark}

\section{Examples}
If $\nu$ is a Bernoulli measure on $X$, then $(X,\nu, S_\infty)$ is a system of type II$_1$. We have shown
in Section 2 that such a system is AT.

Recall that an ergodic and nonsingular action of a countable group $G$ on a Lebesgue space $(Y,\nu)$ is of type III if there is no $G$-invariant measure equivalent to $\nu$. Moreover, the ergodic $G$-space $(Y,\nu)$ is of type III$_\lambda$, $0\leq \lambda\leq 1$, if its associated flow is the periodic flow on the interval $[0,-\log \lambda]$ for $0<\lambda<1$, is the trivial flow on a singleton for $\lambda=1$, and is non-transitive for $\lambda=0$. Notice that the type of an ergodic measurable dynamical system depends only on its orbit equivalence class.

In this section we give examples of product measures on $X$ with respect to which the action of $S_\infty$ is AT and of type III.

\begin{example} 
\end{example}
\noindent For $0<\lambda<1$, let $\lambda_n=\lambda^{2^n}$, for $n\geq 0$ and $(L_n)_{n\geq 0}$ be a sequence of positive integers such that
$$\sup\{ L_n\lambda_n :  n\geq 0\}=\infty.$$
Let $\nu=\nu(L_n, \lambda_n)$ be the corresponding product measure on $X=\prod\limits_{k\geq 0}\{0,1\}$, as in Section 3. If $\mathcal{R}$ denotes the tail equivalence relation on $(X,\nu)$, then the equivalence $\mathcal{S}$ induced by the action of $\mathcal{S}_\infty$ is a subequivalence of $\mathcal{R}$.

By Theorem \ref{the}, $(X,\nu,S_\infty)$ is AT and of type III (see \cite{SV}, \cite{BP} or \cite{GM}). Moreover, as $(X,\nu,\mathcal{R})$ is of type III$_0$, then $(X,\nu,\mathcal{S}_\infty)$ is also of type III$_0$ (the associated flow of $(X,\nu,\mathcal{R})$ is a factor of the flow associated to $(X,\nu,\mathcal{S}_\infty)$).


\begin{example}
\end{example}
\noindent Let $0<\lambda<1$ be fixed and let $(k_n)_{n\geq 0}$ be an increasing sequence of positive integers with $k_0=0$ such that
$$\sum_{n\geq 1} (k_n-k_{n-1})\lambda^{2^n}=\infty.$$
Let $X=\prod_{k\geq 0}\{0,1\}$ and $\nu=\otimes_{k\geq 0} \nu_k$ be the probability measure defined by
\begin{align*}
\nu_{2k}(0)&=\frac{1}{1+\lambda}, \ \ \  \nu_{2k}(1)=\frac{\lambda}{1+\lambda},\ \  k\geq 0
\end{align*}
and
\begin{align*}
\nu_{2k+1}(0)&=\frac{1}{1+\lambda^{2^n+1}}, \ \ \ \nu_{2k+1}(1)=\frac{\lambda^{2^n+1}}{1+\lambda^{2^n+1}}, \  k_{n-1}\leq k< k_n, \  n\geq 1.
\end{align*}
The dynamical system $(X, \nu, S_\infty)$ is AT by Proposition \ref{pro}. Since $\sum_{n\geq 0} (k_n-k_{n-1})\lambda^{2^n}=\infty$, it follows from \cite{SV} or \cite{GM}, that $(X, \nu, S_\infty)$ is of type III.
Let $T(X,\nu,S_\infty)$ be the invariant $T$ of the ergodic system $(X,\nu,S_{\infty})$. Recall that $T(X,\nu,S_\infty)$ is equal to the Connes invariant $T$ of the Krieger factor associated to the system $(X, \mu, S_\infty)$ and to the $L^\infty$-point spectrum of the associated flow (see for example \cite{HO} or \cite{T}).

Then, following \cite{GM}, we have $$T(X,\nu, S_\infty)\supseteq \left\{\frac{2k\pi}{2^{n}\log \lambda}, n\geqslant 0, k\in \mathbb{Z}\right\},$$
which implies that $(X, \nu, S_\infty)$ is a system of type III$_0$. Note that in this case the odometer defined on $(X,\nu)$ is of type III$_\lambda$

This example shows that there exist product probability measures $\otimes_{k\geq 1}\nu_k$ on $X$ such that the corresponding sequence $(\lambda_k)_{k\geq 1}$ has a nonzero limit point and such that $(X, \nu, S_\infty)$ is AT and of type III$_0$.

\subsection*{Acknowledgement. }T. Giordano was partially supported by NSERC Discovery Grant. R. B. Munteanu was supported by the grants of the Romanian Ministry of Education, CNCS - UEFISCDI, project number
PN-II-RU-PD-2012-3-0533 and project number PN-II-RU-TE-2014-4-0669.


\bigskip


\begin{thebibliography}{26}



\bibitem{L}
E. H.~El Abdalaoui and M.~Lemanczyk. Approximately transitive dynamical systems and simple spectrum. {\it
Arch. Math.} \textbf{97} (2011),
187--197.

\bibitem{AEG}
S. Adams, G. A. Elliott and T. Giordano. Amenable actions of groups, {\it
Trans. Amer. Math. Soc.} \textbf {344} (1994), no. 2, 803--822.

\bibitem{AP}
D. Aldous, J. Pitman. On the zero-one law for exchangeable
events. {\it Ann. Probab.} \textbf{7} (1979), no~1, 704--723.

\bibitem{AW} 
H. Araki and E. J. Woods. A classification of factors. {\it Publ. Res. Inst. Math.
Sci. Ser. A} \textbf{4} (1968), 51--130.

\bibitem{B} 
B. M. Baker. Free states of the gauge invariant canonical anticommutation relations. {\it Trans. Amer. Math. Soc.} \textbf{ 237} (1978), 35--61.

\bibitem{BP} 
B. M. Baker and R. T. Powers. Product states on the
gauge invariant and rotationally invariant CAR algebras. {\it J.
Operator Theory} \textbf{10} (1983), 365--393.

\bibitem{CW1} 
A. Connes and E. J. Woods. Approximately transitive flows and ITPFI factors. {\it Ergod. Th. Dynam. Sys.} {\textbf 5} (1985), 203--236.

\bibitem{CW2} 
A. Connes and E. J. Woods. Hyperfinite von Neumann algebras and Poisson
boundaries of time dependent random walks. {\it Pacific J. Math.} \textbf{137} (1989), 225--243.

\bibitem{DQ}
A.~H.~Dooley and A.~Quas. Approximate transitivity for zero-entropy systems. {\it Ergod. Theory  Dynam. Sys.} \textbf{25} (2005), 443--453.

\bibitem{EG1} 
G. A. Elliott and T.  Giordano. Amenable actions of discrete groups. {\it Ergod. Th.
and Dynam. Sys.} \textbf{13} (1993), 289--318.

\bibitem{EG2} 
G. A. Elliott and T. Giordano. Every approximately transitive amenable action
of a locally compact group is a Poisson boundary. {\it C.R. Acad. Sci. Canada} \textbf{21} (1999), 9--15.

\bibitem{GH} 
T. Giordano and D. Handelman. Matrix-valued random walks and variations on
property AT. {\it Munster J. Math.} \textbf{1}  (2008), 15--72.

\bibitem{GM} T. Giordano and R. B. Munteanu. Von Neumann Algebras Arising as Fixed Point Algebras  under Xerox Type Actions, {\it J. of Operator Theory} \textbf  {72} (2014), 343-369.

\bibitem{HO}
T. Hamachi and M. Osikawa. Ergodic groups of automorphisms and Krieger's Theorem. {\it Sem. Math. Sci. Keio University} \textbf{3} (1981).

\bibitem{HS} 
E. Hewitt and K. Stromberg. {\it Real and abstract analysis.} Berlin-Heidelberg-New
York, Springer (1965).

\bibitem{K} 
W. Krieger. On ergodic flows and isomorphisms of factors. {\it Math. Ann.} \textbf{223} (1976), 19--70.

\bibitem{SV} 
S. Stratila and D. V. Voiculescu. {\it Representations of
AF-algebras and of the group $ U(\infty)$.}  Springer Lect. Notes
Math. \textbf{486}, Springer-Verlag, Berlin  (1975).

\bibitem{S} 
C. E. Sutherland. Notes on orbit equivalence: Krieger's Theorem, {\it Lecture Notes Series} \textbf{23} (1976), Oslo.

\bibitem{T} 
M. Takesaki. {\it Theory of Operator Algebras  II.} Encyclopaedia of Mathematical Sciences,  \textbf{124}, Springer-Verlag,
New York (2001).
\end{thebibliography}
\end{document}